\documentclass[a4note,11pt]{article}
 \usepackage[english]{babel}
 \usepackage{graphicx}
\usepackage[margin=1.54in]{geometry}
 \usepackage[utf8]{inputenc}
 \usepackage[T1]{fontenc}
 \usepackage{lmodern}
 \usepackage[normalem]{ulem}
 \usepackage{verbatim}
 \usepackage{bbm}
 \usepackage{ntheorem}
 \usepackage{stmaryrd}
 \usepackage{amsmath}
\usepackage{dsfont}
 \usepackage{amssymb}
\usepackage{hyperref}
\usepackage{enumitem}

\usepackage{graphicx}
\usepackage{subcaption}

\usepackage{pgf,tikz}
\usepackage{mathrsfs}
\usetikzlibrary{arrows}

\usepackage{color}

\theoremseparator{.}
\newtheorem{theorem}{Theorem}[section]
\newtheorem{thm}[theorem]{Theorem}

\newtheorem{coro}[theorem]{Corollary}
\newtheorem{lemma}[theorem]{Lemma}

\newtheorem{prop}[theorem]{Proposition}

\theoremstyle{definition}
\newtheorem{question}[theorem]{Question}

\newenvironment{rem}[1][Remark.]{\begin{trivlist}
\item[\hskip \labelsep {\itshape #1}]}{\end{trivlist}}

\def\sqw{\hbox{\rlap{\leavevmode\raise.3ex\hbox{$\sqcap$}}$%
\sqcup$}}

\newcommand{\N}{\ensuremath{{\mathbb N}}}

\newcommand{\Z}{\ensuremath{\mathbb Z}}

\newcommand{\Pp}{\mathbb{P}}

\newcommand{\defini}{\textbf}

\def\acts{\curvearrowright}

\newcommand{\saut}{\vspace{0.25cm}}

\newenvironment{proof}{  
    \vspace*{-.4em}  {\it Proof.}%
}{
    \hfill\sqw\vspace*{.5em}
}

\newenvironment{proofpu}{  
    \vspace*{-.4em}  {\it Proof of Theorem~\ref{thm:pu}.}%
}{
    \hfill\sqw\vspace*{.5em}
}

\setlist[itemize]{noitemsep, nolistsep}
\setlist[enumerate]{noitemsep, nolistsep}

\author{Sébastien \textsc{Martineau}\footnote{Université Paris-Sud, {\href{mailto:sebastien.martineau@u-psud.fr}{\nolinkurl{sebastien.martineau@u-psud.fr}}}}~~and Franco \textsc{Severo}\footnote{Institut des Hautes \'Etudes Scientifiques, {\href{mailto:severo@ihes.fr}{\nolinkurl{severo@ihes.fr}}}}}
\title{Strict monotonicity of percolation thresholds under covering maps}
\begin{document}

\maketitle

\begin{abstract}
 We answer a question of Benjamini and Schramm by proving that under reasonable conditions, quotienting a graph strictly increases the value of its percolation critical parameter $p_c$. More precisely, let $\mathcal{G}=(V,E)$ be a quasi-transitive graph with $p_c(\mathcal{G})<1$, and let $G$ be a nontrivial group that acts freely on $V$ by graph automorphisms. Assume that $\mathcal{H}:=\mathcal{G}/G$ is quasi-transitive. Then one has $p_c(\mathcal{G})<p_c(\mathcal{H})$.

We provide results beyond this setting: we treat the case of general covering maps and provide a similar result for the uniqueness parameter $p_u$, under an additional assumption of boundedness of the fibres. The proof makes use of a coupling built by lifting the exploration of the cluster, and an exploratory counterpart of Aizenman-Grimmett's essential enhancements.
\end{abstract}

\vspace{0.1cm}

\section{Introduction}

Bernoulli percolation is a simple model for problems of propagation in porous media that was introduced in 1957 by Broadbent and Hammersely \cite{broadbenthammersley}: given a graph $\mathcal{G}$ and a parameter $p\in [0,1]$, erase each edge independently with probability $1-p$. Studying the connected components of this random graph (which are referred to as \defini{clusters}) has been since then an active field of research: see the books \cite{grimmettbook, lyonsperesbook}. A prominent quantity in this theory is the so-called \defini{critical parameter} $p_c(\mathcal{G})$, which is characterised by the following dichotomy: for every $p<p_c(\mathcal{G})$, there is almost surely no infinite cluster, while for every $p>p_c(\mathcal{G})$, there is almost surely at least one infinite cluster.

Originally, the main focus was on the Euclidean lattice $\Z^d$.
In 1996, Benjamini and Schramm initiated the systematic study of Bernoulli percolation on more general graphs, namely quasi-transitive graphs \cite{bspercobeyond}. A graph is \defini{quasi-transitive} (resp. \defini{transitive}) if the action of its automorphism group on its vertices yields finitely many orbits (resp. a single orbit). Intuitively, a graph is quasi-transitive if it has finitely many types of vertices, and transitive if all the vertices look the same.
The paper \cite{bspercobeyond} contains, as its title suggests, many questions and a few answers: in their Theorem~1 and Question~1, they investigate the monotonicity of $p_c$ under quotients. Their Question~1 is precisely the topic of the present paper. It goes as follows.

\paragraph{Setting of \cite{bspercobeyond}} Let $\mathcal{G}=(V,E)$ be a locally finite \emph{connected} graph. Let $G$ be a group acting on $V$ by graph automorphisms. 
A vertex of the \defini{quotient graph} $\mathcal{G}/G$ is an orbit of $G\acts V$, and two distinct orbits are connected by an edge if and only if there is an edge of $\mathcal{G}$ intersecting both orbits.

Theorem~1 of \cite{bspercobeyond} asserts that $p_c(\mathcal{G})\leq p_c(\mathcal{G}/G)$. It is proved by lifting the exploration of a spanning tree of the cluster of the origin from $\mathcal{G}/G$ to $\mathcal{G}$. They then ask the following natural question.
Recall that a group action $G\acts X$ is \defini{free} if the only element of $G$ that has a fixed point is the identity element: $$\forall g\in G\backslash \{1\},~\forall x\in X,~g x \not= x.$$

\begin{question}[Benjamini-Schramm]
\label{quest:bs}
Consider the quotient graph $\mathcal{H}:= \mathcal{G}/G$. Assume that $p_c(\mathcal{G})<1$, $G\not=\{1\}$ acts freely on $V(\mathcal{G})$, and both $\mathcal{G}$ and $\mathcal{H}$ are quasi-transitive. Is it necessarily the case that $p_c(\mathcal{G})<p_c(\mathcal{H})$?
\end{question}

The main result of the {present} paper is a positive answer to this question: see Corollary~\ref{coro:answer}. We use an exploratory version of Aizenman-Grimmett's essential enhancements \cite{aizenmangrimmett}, and build a coupling between $p$-percolation on $\mathcal{G}$ and enhanced percolation on $\mathcal{H}$ by lifting the exploration of the cluster of the origin. The part of our work devoted to essential enhancements (Section~\ref{sec:prooftwo}) follows the Aizenman-Grimmett strategy, thus making crucial use of certain differential inequalities, see also \cite{menshikovstrict}. Our coupling (Section~\ref{sec:proofone}) improves on that used in \cite{bspercobeyond}.

We also address in Theorem~\ref{thm:pu} a similar question for the \defini{uniqueness parameter} $p_u$. Recall that given a \emph{quasi-transitive graph} $\mathcal{G}$, the number of infinite connected components for Bernoulli percolation of parameter $p$ takes an almost sure value $N_\mathcal{G}(p)\in \{0,1,\infty\}$, and that the following monotonicity property holds: $\forall~ p<q,~N_\mathcal{G}(p)=1\implies N_\mathcal{G}(q)=1$, see \cite{Sch99}. One thus defines $p_u(\mathcal{G}):= \inf \{p\in [0,1]:N_\mathcal{G}(p)=1\}$.

\begin{center}
\setlength{\unitlength}{1mm}
\thicklines
\begin{picture}(100,10)
\put(0,0){
   \line(1,0){100}
}

\put(33,6){$p_c$}
\put(33,-1.5){
   \line(0,1){3}
}

\put(67,6){$p_u$}
\put(67,-1.5){
   \line(0,1){3}
}

\put(0,6){$0$}
\put(0,-1.5){
   \line(0,1){3}
}

\put(100,6){1}
\put(100,-1.5){
   \line(0,1){3}
}

\put(12,3){$N_\mathcal{G} = 0$}

\put(45,3){$N_\mathcal{G} = \infty$}

\put(79,3){$N_\mathcal{G} = 1$}
\end{picture}
\end{center}

Let us mention that a theorem quite similar to our Corollary~\ref{coro:answer} has already been obtained for the connective constant instead of $p_c$. See Theorem~3.8 in \cite{glstrict}.

\section{Definitions and results}

To avoid any ambiguity, let us review the relevant vocabulary.

\vspace{-0.2cm}

\paragraph{Convention} Graphs are taken to be non-empty, locally finite (every vertex has finitely many neighbours) and \emph{connected}. Subgraphs (e.g.~percolation configurations) may not be connected. Unless otherwise stated, our graphs are taken to be simple (no multiple edges, no self-loops, edges are unoriented). A graph $\mathcal{G}$ may be written in the form $(V,E)$, where $V=V(\mathcal{G})$ denotes its set of vertices and $E=E(\mathcal{G})$ its set of edges. An edge is a subset of $V$ with precisely two elements. The \defini{degree} of a vertex is its number of neighbours. Graphs are endowed with their respective graph distance, denoted by $d$. Finally, percolation is taken to mean Bernoulli \emph{bond} percolation, but our proofs can be adapted to Bernoulli \emph{site} percolation. 

\vspace{0.2cm}

In Question~\ref{quest:bs}, the graphs $\mathcal{G}$ and $\mathcal{H}$ are related via the quotient map $\pi : x \mapsto Gx$. This map is a \defini{weak covering map}, meaning that it is 1-Lipschitz for the graph distance and that  it has the \defini{weak lifting property}: for every $x\in V(\mathcal{G})$ and every neighbour $u$ of $\pi(x)$, there is a neighbour of $x$ that is mapped to $u$. This fact does not use the freeness of the action of $G$ or quasi-transitivity. {Weak covering maps are by definition able to lift edges, but it turns out they can also lift trees}, meaning that for every subtree of the target space and every vertex in the preimage of the tree, there is a lift of the tree that contains this vertex. Recall that given a subtree\footnote{i.e.~a tree with $V(\mathcal{T})\subset V(\mathcal{H})$ and $E(\mathcal{T})\subset E(\mathcal{H})$.} $\mathcal{T}$ of $\mathcal{H}$, a \defini{lift} of $\mathcal{T}$ is a subtree $\mathcal{T}'$ of $\mathcal{G}$ such that $\pi$ induces a graph isomorphism from $\mathcal{T}'$ to $\mathcal{T}$, i.e.~it induces well-defined bijections from $V(\mathcal{T}')$ to $V(\mathcal{T})$ and from $E(\mathcal{T}')$ to $E(\mathcal{T})$.

The map $\pi$ satisfies a second property, namely \defini{disjoint tree-lifting}: if $\mathcal{T}$ is a subtree of $\mathcal{H}$ and if $x$ and $y$ are distinct vertices of $\mathcal{G}$ such that $\pi(x)=\pi(y)$ belongs to $V(\mathcal{T})$, then one can find two \emph{vertex-disjoint} lifts of $\mathcal{T}$ such that one of them contains $x$ and the other $y$. This fact uses the freeness of $G$, and is established in Lemma~\ref{lem:disjointtrees}.

Finally, the map $\pi$ \defini{has tame fibres}: there is some $R$ such that for every $x\in V(\mathcal{G})$, there is some $y\in V(\mathcal{G})$ satisfying $\pi(x)=\pi(y)$ and $0<d(x,y)\leq R$. See Lemma~\ref{lem:tamefibres}.

\vspace{0.2cm}

{It turns out} that these three properties of $\pi$ suffice to prove strict inequality, so that there is actually no need for group actions and quasi-transitivity.

\begin{thm}
\label{thm:main}
Let $\mathcal{G}$ and $\mathcal{H}$ be graphs of bounded degree. Assume that there is a weak covering map $\pi : V(\mathcal{G})\to V(\mathcal{H})$ with tame fibres and the disjoint tree-lifting property.
If $p_c(\mathcal{G})<1$, then one has $p_c(\mathcal{G})<p_c(\mathcal{H})$.
\end{thm}

\begin{coro}
\label{coro:answer}
Let $\mathcal{G}$ be a graph. Let $G\neq \{1\}$ be a group acting on $V(\mathcal{G})$ by graph automorphisms, and let $\mathcal{H}$ denote the quotient graph $\mathcal{G}/G$.
Assume that the following conditions hold:
\begin{enumerate}
\item \label{freeness}the action $G\acts V(\mathcal{G})$ is free,
\item \label{gqtrans}$\mathcal{G}$ is quasi-transitive,
\item \label{hqtrans}$\mathcal{H}$ is quasi-transitive,
\item \label{pcnontriv}$p_c(\mathcal{G})<1$.
\end{enumerate}
Then one has $p_c(\mathcal{G})<p_c(\mathcal{H})$.
\end{coro}

\begin{rem}
By using the techniques of \cite{cantor}, one can deduce from Corollary~\ref{coro:answer} and  \cite[exercice p.~4]{hutchcroftlocality} that when $\mathcal{G}$ ranges over Cayley graphs of 3-solvable groups, $p_c(\mathcal{G})$ takes uncountably many values. Actually, one gets at least a Cantor set of such values. This is optimal in the following sense: there are only countably many 2-solvable finitely generated groups (see Corollary 3 in \cite{hall1954}), hence only countably many Cayley graphs of such groups. The same result without the solvability condition has been obtained previous to \cite{cantor} by Kozma \cite{kozmapersonal}, by working with graphs of the form $\mathcal{G}\star \mathcal{G}$. 
\end{rem}

Theorem~\ref{thm:main} yields a second corollary. Say that a map $\pi:V(\mathcal{G})\to V(\mathcal{H})$ is a \defini{strong covering map} if it is 1-Lipschitz for the graph distance and has the \defini{strong lifting property}: for every $x\in V(\mathcal{G})$, for every neighbour $u$ of $\pi(x)$, there is a \emph{unique} neighbour of $x$ that maps to $u$. {Recall that for many authors, the definition of a ``covering map'' is taken to be even stricter: a \defini{classical covering map} is a \emph{graph homomorphism} with the strong lifting property.}

{By} Theorem~\ref{thm:main} and Lemma~\ref{lem:strong}, the following result holds.

\begin{coro}
\label{coro:strong}
Let $\mathcal{G}$ and $\mathcal{H}$ be graphs of bounded degree. Assume that there is a strong covering map $\pi : V(\mathcal{G})\to V(\mathcal{H})$ with tame fibres.

If $p_c(\mathcal{G})<1$, then one has $p_c(\mathcal{G})<p_c(\mathcal{H})$.
\end{coro}

We also have the following results regarding the value of $p_u$. 
Say that a weak covering map $\pi : V(\mathcal{G})\to V(\mathcal{H})$ has \defini{bounded fibres} if there is some $K$ such that
$$
\forall x,y\in V(\mathcal{G}),~\pi(x)=\pi(y)\implies d(x,y)\leq K.
$$
The following two theorems are, respectively, the $p_u$ counterparts of Theorem 1 from \cite{bspercobeyond} and Theorem \ref{thm:main} above.
\begin{thm}
	\label{thm:nonstrictpu}
	Let $\mathcal{G}$ and $\mathcal{H}$ be quasi-transitive graphs. Assume that there is a weak covering map $\pi : V(\mathcal{G})\to V(\mathcal{H})$ with bounded fibres. 
	
	Then one has $p_u(\mathcal{G})\leq p_u(\mathcal{H})$.
\end{thm}

\begin{thm}
\label{thm:pu}
Let $\mathcal{G}$ and $\mathcal{H}$ be quasi-transitive graphs. Assume that there is a non-injective weak covering map $\pi : V(\mathcal{G})\to V(\mathcal{H})$ with bounded fibres and the disjoint tree-lifting property.

If $p_u(\mathcal{G})<1$, then one has $p_u(\mathcal{G})<p_u(\mathcal{H})$.
\end{thm}

The next corollary follows from Theorem~\ref{thm:pu} and Lemma~\ref{lem:bounded}.

\begin{coro}
\label{coro:answerpu}
Let $\mathcal{G}$ be a graph. Let $G\neq \{1\}$ be a \emph{finite} group acting on $V(\mathcal{G})$ by graph automorphisms, and let $\mathcal{H}$ denote the quotient graph $\mathcal{G}/G$.
Assume that the following conditions hold:
\begin{enumerate}
\item \label{freenesspu}the action $G\acts V(\mathcal{G})$ is free,
\item \label{gqtranspu}$\mathcal{G}$ is quasi-transitive,
\item \label{hqtranspu}$\mathcal{H}$ is quasi-transitive,
\item \label{pcnontrivpu}$p_u(\mathcal{G})<1$.
\end{enumerate}
Then one has $p_u(\mathcal{G})<p_u(\mathcal{H})$.
\end{coro}

\begin{coro}
\label{coro:strongpu}
Let $\mathcal{G}$ and $\mathcal{H}$ be quasi-transitive graphs. Assume that there is a non-injective strong covering map $\pi : V(\mathcal{G})\to V(\mathcal{H})$ with bounded fibres. 

If $p_u(\mathcal{G})<1$, then one has $p_u(\mathcal{G})<p_u(\mathcal{H})$.
\end{coro}

Our proofs can be made explicit in that they actually yield quantitative (but poor) lower bounds on the differences $p_c(\mathcal{H})-p_c(\mathcal{G})$ and $p_u(\mathcal{H})-p_u(\mathcal{G})$.

\paragraph{Structure of the paper} The remaining of the paper is organised as follows.
Section~\ref{sec:hyp} discusses the hypotheses of our results. Section~\ref{sec:proofmain} exposes the strategy of proof of Theorem~\ref{thm:main}, reducing it to the proof of Propositions~\ref{propone} and \ref{proptwo}. They are respectively established in Sections~\ref{sec:proofone} and \ref{sec:prooftwo}. Section~\ref{sec:corotothm} explains why the corollaries follow from the theorems. Finally, Section~\ref{sec:pu} is devoted to the proof of Theorems~\ref{thm:nonstrictpu} and \ref{thm:pu}.

\section{On the hypotheses of our results}
\label{sec:hyp}

None of the four assumptions of Corollary~\ref{coro:answer} can be removed.
This is clear for Hypothesis~\ref{pcnontriv}. For Hypothesis~\ref{freeness}, take $\mathcal{G}_0$ to be a quasi-transitive graph with $p_c<1$ --- e.g. the square lattice or the 3-regular tree --- and define $\mathcal{G}$ by setting: $$V(\mathcal{G}):=V(\mathcal{G}_0)\times \{0,1,2\},~~E(\mathcal{G}):=\left(E(\mathcal{G}_0)\times\{0\}\right)\cup \left(V(\mathcal{G})\times\{\{0,1\},\{0,2\}\}\right),$$
where one identifies $E(\mathcal{G}_0)\times\{0\}$ with $\{\{(x,0),(y,0)\}:\{x,y\}\in E(\mathcal{G}_0)\}$ and $V(\mathcal{G})\times\{\{0,1\},\{0,2\}\}$ with $\{\{(x,0),(x,i)\}:x\in V(\mathcal{G}_0),~i\in\{1,2\}\}$. It suffices then to take $G:= \Z/2\Z$ with the non-trivial element acting via $(x,i)\mapsto (x,\sigma(i))$, where $\sigma$ is the $(12)$-transposition. Notice that if $\mathcal{G}_0$ is taken to be amenable (e.g. the square lattice), then $p_u(\mathcal{G})=p_c(\mathcal{G})$ and $p_u(\mathcal{H})=p_c(\mathcal{H})$, so that Hypothesis~\ref{freenesspu} is also necessary in Corollary~\ref{coro:answerpu}. See \cite{burtonkeane}.

For Hypothesis~\ref{gqtrans}, take once again $\mathcal{G}_0$ to be a quasi-transitive graph with $p_c<1$, and pick some vertex $o$ in $\mathcal{G}_0$. The graph $\mathcal{G}$ is defined by taking two disjoint copies of $\mathcal{G}_0$ and putting an additional edge between the two copies of $o$. The group $G:=\Z/2\Z$ acts by swapping copies.

As for Hypothesis~\ref{hqtrans}, take $\mathcal{G}$ to be the square lattice $\Z^2$, and $G$ to be $\Z/2\Z$ acting via the reflection $(x,y) \mapsto (x,1-y)$.

\saut

Still, we do not know what happens if freeness is relaxed to the absence of trivial $G$-orbit.

\begin{question}
\label{quest:ours}
Let $\mathcal{G}$ be a graph. Let $G$ be a group acting on $V(\mathcal{G})$, and let $\mathcal{H}$ denote the quotient graph $\mathcal{G}/G$.
Assume that the following conditions hold:
\begin{itemize}
\item $\forall x\in V(\mathcal{G}),~\exists g\in G,~gx\not=x$,
\item $\mathcal{G}$ is quasi-transitive,
\item $\mathcal{H}$ is quasi-transitive,
\end{itemize}
Is it necessarily the case that $p_c(\mathcal{G})<1$ implies $p_c(\mathcal{G})<p_c(\mathcal{H})$? If we assume further that $G$ is finite, is it necessarily the case that $p_u(\mathcal{G})<1$ implies $p_u(\mathcal{G})<p_u(\mathcal{H})$?
\end{question}
\begin{rem}
An interesting particular case (which we also do not know how to solve) is when $G$ is normal in a quasi-transitive subgroup of $\mathsf{Aut}(\mathcal{G})$. In that  setting, $\mathcal{H}$ is automatically quasi-transitive, and the map $\pi$ always has tame fibres.
\end{rem}

As for Theorem~\ref{thm:main} and Corollary~\ref{coro:strong}, notice that the assumption that fibres are tame cannot be replaced by non-triviality of the fibres {(namely $\forall u\in V(\mathcal{H}),~|\pi^{-1}(\{u\})|\not=1$), even if $\pi$ is taken to be a \emph{classical} covering map}. Indeed, take $\mathcal{H}$ to be a graph with bounded degree and $p_c<1$, and pick some edge $e$ in $\mathcal{H}$. To define $\mathcal{G}$, start with two copies of $\mathcal{H}$, and denote by $\{x,y\}$ and $\{x',y'\}$ the two copies of $e$. Then replace these two edges by $\{x',y\}$ and $\{x,y'\}$, thus yielding a connected graph. Take $\pi$ to be the natural projection from $\mathcal{G}$ to $\mathcal{H}$.

\saut

We do not know how to answer the following question, which investigates a generalisation of Theorem~\ref{thm:main}/Corollary~\ref{coro:strong}.

\begin{question}
\label{quest:ourstwo}
Let $\mathcal{G}$ and $\mathcal{H}$ be graphs of bounded degree. Assume that there is a weak covering map $\pi : V(\mathcal{G})\to V(\mathcal{H})$ with tame fibres.

If $p_c(\mathcal{G})<1$, is it necessarily the case that $p_c(\mathcal{G})<p_c(\mathcal{H})$?
\end{question}

\begin{question}
\label{quest:ourstwo}
Let $\mathcal{G}$ and $\mathcal{H}$ be quasi-transitive graphs. Assume that there is a weak covering map $\pi : V(\mathcal{G})\to V(\mathcal{H})$ with tame and bounded fibres.

If $p_u(\mathcal{G})<1$, is it necessarily the case that $p_u(\mathcal{G})<p_u(\mathcal{H})$?
\end{question}

\saut

Finally, notice that one cannot remove the finiteness assumption from Theorem~\ref{thm:pu} and its corollaries. Indeed, without this assumption, it is even possible to have the strict inequality in the \emph{reverse} direction. The following example shows that this is easy to obtain if one further relaxes the assumption that $p_u(\mathcal{G})<1$: if one takes $\mathcal{G}$ to be the $2d$-regular tree and $\mathcal{H}$ to be the $d$-dimensional hypercubic lattice, for some $d\geq 2$, then we have $$p_u(\mathcal{H})=p_c(\mathcal{H})<1=p_u(\mathcal{G}).$$
If one does not want to relax the assumption that $p_u(\mathcal{G})<1$, one can take $d$ to be large enough, $\mathcal{G}_d$ to be the product of the $2d$-regular tree and the bi-infinite line, and $\mathcal{H}_d$ to be the $(d+1)$-dimensional hypercubic lattice. Indeed, $p_u(\mathcal{G}_d)\sim \frac{1}{\sqrt{d}}$ but $p_u(\mathcal{H}_d)=p_c(\mathcal{H}_d)\sim \frac{1}{2d}$, see respectively \cite{GN90} and \cite{Kes90}.

\section{Proof of Theorem~\ref{thm:main}}
\label{sec:proofmain}

Let $\mathcal{G}$, $\mathcal{H}$ and $\pi$ be as in Theorem~\ref{thm:main}.
Let $r$ be a positive integer. Pick a root $o$ in $\mathcal{H}$, and some $o'\in \pi^{-1}(\{o\})$.

\paragraph{Notation} Given a graph $(V,E)$, the ball of centre $x$ and radius $r$ is $B_r(x):=\{y\in V:d(x,y)\leq r\}$. It is considered as a set of vertices, but it may also be considered as a graph --- with the structure the ambient graph induces on it. For $r\in \N$, the sphere of centre $x$ and radius $r$ is $S_r(x):=\{y\in V:d(x,y)=r\}$. We also set $S_{r+\frac{1}{2}}(x):=\{e\in E: e\cap S_r(x)\not=\varnothing\text{ and }e\cap S_{r+1}(x)\not=\varnothing\}$.

\vspace{0.3cm}

We are going to construct a random subset $\mathcal{C}_0$ of $V(\mathcal{H})$ which will be a ``strict enhancement'' of the cluster of $o$ in a $p$-percolation model on $\mathcal{H}$. Given a configuration $(\omega,\alpha)\in \{0,1\}^{E(\mathcal{H})}\times\{0,1\}^{V(\mathcal{H})}$, we define inductively a sequence $(C_n)_{n\geq 0}$ of subsets of $V(\mathcal{H})$ as follows. We sometimes identify $\omega$ with the subset of edges $\{e:~\omega_e=1\}$ or the subgraph of $\mathcal{H}$ associated with it. Set $C_0 :=\{o\}$. For $n\geq 0$, let $C_{2n+1}$ be the union of the $\omega$-clusters of the vertices of $C_{2n}$. Then let $C_{2n+2}$ be the union of $C_{2n+1}$ and the vertices $v$ such that there is some $u\in C_{2n+1}$ satisfying the following conditions:
\begin{enumerate}
	\item $d(u,v)=r+1$,
	\item $\omega_e=1$ for all edges $e$ in $B_r(u)$,
	\item $\alpha_u=1$.
\end{enumerate}
The sequence of sets $(C_n)$ is non-decreasing, and we define $\mathcal{C}_o=\mathcal{C}_o(\omega,\alpha):= \bigcup_n C_n$. Given $p,s\in [0,1]$, the distribution of the random variable $\mathcal{C}_{o}(\omega,\alpha)$ under the probability measure $\mathbb{P}_{p,s}:=\mathsf{Ber}(p)^{\otimes E(\mathcal{H})}\otimes \mathsf{Ber}(s)^{\otimes V(\mathcal{H})}$ is denoted by $\mathcal{C}_{\mathcal{H}}^{p,s}(o)$. In a similar way, we can define $\mathcal{C}_{A}=\mathcal{C}_{A}(\omega,\alpha)$ --- and its distribution under $\mathbb{P}_{p,s}$, denoted by $\mathcal{C}_{\mathcal{H}}^{p,s}(A)$ --- by considering the same process but initialising it with $C_0=A$.
We also set $\mathcal{C}_\mathcal{G}^p(A)$ to be the distribution of the cluster of $A$ in bond percolation of parameter $p$ on $\mathcal{G}$.
\begin{rem}
Note that $\mathcal{C}_o(\omega,\alpha)$ does not coincide with the cluster of $o$ for the following model: declare an edge $e$ to be open if ``$e$ is $\omega$-open or there is a vertex $u$ such that $e\in S_{r+\frac{1}{2}}(u)$, all the edges in $B_r(u)$ are $\omega$-open and $\alpha_u=1$''. This would be an instance of the \textit{classical} enhancement introduced by Aizenman and Grimmett --- see \cite{aizenmangrimmett}. Indeed, the model we consider here is an \textit{exploratory version} of their model, and the former is stochastically dominated by the latter. For example, in our model the assertion $v\in \mathcal{C}_u(\omega,\alpha)$ does not necessarily imply $u\in \mathcal{C}_v(\omega,\alpha)$.
\end{rem}

We will prove the following two propositions. The proof of Proposition~\ref{propone} proceeds by lifting some exploration process from $\mathcal{H}$ to $\mathcal{G}$: in that, it is similar to the proof of Theorem~1 of \cite{bspercobeyond}. The proof of Proposition~\ref{proptwo} uses an exploratory variation of the techniques of Aizenman and Grimmett \cite{aizenmangrimmett}. Even though essential enhancements are delicate in general \cite{essrev}, it turns out that our particular enhancement \emph{can} be handled for general graphs, even for site percolation.

\begin{prop}
\label{propone}There is a choice of $r\geq 1$ such that the following holds: for every $\varepsilon >0$, there is some $s\in (0,1)$ such that for every $p\in [\varepsilon, 1]$, $\mathcal{C}_{\mathcal{H}}^{p,s}(o)$ is stochastically dominated\footnote{There is a coupling such that the $(\mathcal{H},p,s)$-cluster is a subset of the $\pi$-image of the $(\mathcal{G},p)$-cluster.} by $\pi\left(\mathcal{C}_{\mathcal{G}}^{p}(o')\right)$.
\end{prop}

\begin{prop}
\label{proptwo}
Assume further that $p_c(\mathcal{H})<1$.
Then, for any choice of $r\geq 1$, the following holds: for every $s\in (0,1]$, there exists $p_s<p_c(\mathcal{H})$ such that for every $p\in [p_s,1]$, the cluster $\mathcal{C}_\mathcal{H}^{p,s}(o)$ is infinite with positive probability.
\end{prop}

Assuming these propositions, let us establish Theorem~\ref{thm:main}.

\saut

\textit{Proof of Theorem~\ref{thm:main}.} First, notice that if $p_c(\mathcal{H})=1$, then the conclusion holds trivially. We thus assume that $p_c(\mathcal{H})<1$. We pick $r$ so that the conclusion of Proposition~\ref{propone} holds.
Since boundedness of the degree of $\mathcal{H}$ implies that $p_c(\mathcal{H})>0$, we can pick some $\varepsilon$ in $(0,p_c(\mathcal{H}))$.
By Proposition~\ref{propone}, we can pick $s\in (0,1)$ such that for every $p\in [\varepsilon, 1]$, $\mathcal{C}_{\mathcal{H}}^{p,s}(o)$ is stochastically dominated by $\pi\left(\mathcal{C}_{\mathcal{G}}^p(o')\right)$. By Proposition~\ref{proptwo}, there is some $p_s<p_c(\mathcal{H})$ such that for every $p\in [p_s,1]$, the cluster $\mathcal{C}_\mathcal{H}^{p,s}(o)$ is infinite with positive probability. Fix such a $p_s$, and set $p:=\max(p_s,\varepsilon)<p_c(\mathcal{H})$. By definition of $p_s$, the cluster $\mathcal{C}_\mathcal{H}^{p,s}(o)$ is infinite with positive probability. As $p\geq \varepsilon$, the definition of $s$ implies that $\mathcal{C}_{\mathcal{H}}^{p,s}(o)$ is stochastically dominated by $\pi\left(\mathcal{C}_{\mathcal{G}}^p(o')\right)$. As a result, $\pi\left(\mathcal{C}_{\mathcal{G}}^p(o')\right)$ is infinite with positive probability. In particular, $\mathcal{C}_{\mathcal{G}}^p(o')$ is infinite with positive probability, so that $p_c(\mathcal{G})\leq p < p_c(\mathcal{H})$. \vspace{-0.8cm}\begin{flushright}$\square$\end{flushright}

\section{Proof of Proposition~\ref{propone}}
\label{sec:proofone}

\newcommand{\G}{\mathcal{G}}
\newcommand{\Hh}{\mathcal{H}}
\newcommand{\Ghat}{\hat{\mathcal{G}}}
\newcommand{\Hhat}{\hat{\mathcal{H}}}
\newcommand{\phat}{\hat{p}}

The choice of a suitable value of $r$ is given by the following lemma.

\begin{lemma}\label{lempattern}
There is a choice of $r\geq 1$ such that for every $x\in V(\mathcal{G})$, the set $Z=Z(x,r)$ defined as the connected component\footnote{Here $\pi^{-1}(B_r(\pi(x)))\cap B_{3r}(x)$ is seen as endowed with the graph structure induced by $\mathcal{G}$.} of $x$ in $\pi^{-1}(B_r(\pi(x)))\cap B_{3r}(x)$ satisfies that for any $u\in S_{r+1}(\pi(x))$, the fibre $\pi^{-1}(\{u\})$ contains at least two vertices adjacent to $Z$.
\end{lemma}

\begin{proof}
Let $R$ be given by the fact that $\pi$ has tame fibres and set $r:=\lceil \frac{R}{2}\rceil$. Let $x$ be any vertex of $\mathcal{G}$. Take some $y\in V(\mathcal{G})$ such that $\pi(x)=\pi(y)$ and $0<d(x,y)\leq R$. Let $\mathcal{T}$ be a spanning tree of $B_{r+1}(\pi(x))$ obtained by adding first the vertices at distance 1, then at distance 2, etc. As $\pi$ has the disjoint tree-lifting property, one can pick two vertex-disjoint lifts $\mathcal{T}_x$ and $\mathcal{T}_y$ of $\mathcal{T}$ such that $x\in V(\mathcal{T}_x)$ and $y \in V(\mathcal{T}_y)$.

Let $\gamma$ be a geodesic path from $x$ to $y$, thus staying inside $\pi^{-1}(B_r(\pi(x)))$ as $R\leq 2r$. The set $Z'$ consisting in the union of the span of $\gamma$ and $(V(\mathcal{T}_x)\cup V(\mathcal{T}_y))\cap \pi^{-1}(B_r(\pi(x)))$ is a connected subset of $Z(x,r)$: its connectedness results from the choice of the spanning tree $\mathcal{T}$. It thus suffices to prove that for any $u\in S_{r+1}(\pi(x))$, the fibre $\pi^{-1}(\{u\})$ contains at least two vertices adjacent to $Z'$. But this is the case as every such $u$ admits a lift in $\mathcal{T}_x$ and another one in $\mathcal{T}_y$.
\end{proof}

Take $r$ to satisfy the conclusion of Lemma~\ref{lempattern}.
Let $\varepsilon >0$.
Set $M$ and $s$ to be so that the following two conditions hold:
\[\forall e=\{x,y\}\in E(\mathcal{H}),~M\geq |B_r(x)\cup B_r(y)|,\]
\[\forall x \in V(\mathcal{G}),~s\leq \left(1-(1-\varepsilon)^{1/M}\right)^{|E(B_{3r+1}(x))|}.\]
For instance, one may take $M:=D^{r+2}$ and $s:=(1-(1-\varepsilon)^{1/M})^{D^{3r+2}}$, where $D$ stands for the maximal degree of a vertex of $\mathcal{G}$.
Let $p\in [\varepsilon, 1]$.

\saut

We define the multigraph $\Ghat$ as follows: the vertex-set is $V(\mathcal{G})$, the edge-set is $E(\mathcal{G})\times\{1,\dots,M\}$, and $(\{x,y\},k)$ is interpreted as an edge connecting $x$ and $y$. The multigraph $\Hhat$ is defined in the same way, with $\mathcal{H}$ instead of $\mathcal{G}$. The purpose of this multigraph is to allow multiple use of each edge for a bounded number of ``$s$-bonus''. They will play no role as far as $p$-exploration is concerned: concretely, for ``$p$-exploration'', each edge will be considered together with all its parallel copies.

Let $\omega$ be a Bernoulli percolation of parameter $\phat:=1-(1-p)^{1/M}$ on $\Hhat$, so that $\phat$-percolation on $\Hhat$ corresponds to $p$-percolation on $\Hh$. Let $\omega'$ be a Bernoulli percolation of parameter $\phat$ on $\Ghat$ that is independent of $\omega$.  Choose an injection from $E(\Hh)$ to $\N$, so that $E(\Hh)$ is now endowed with a well-ordering. Do the same with $E(\G)$, $V(\G)$ and $V(\Hh)$.

We now define algorithmically an exploration process. This dynamical process will construct edge after edge a Bernoulli percolation $\eta$ of parameter $\phat$ on $\Ghat$ and an $\alpha$ with distribution $\textsf{Ber}(s)^{\otimes V(\mathcal{H})}$. The random variables $\eta$, $\alpha$, and $\omega$ will be coupled in a suitable way, and $\alpha$ will be independent of $\omega$. 

\paragraph{Structure of the process}{In the exploration, edges in $\Ghat$ may get explored in two different ways, called $p$-explored and $s$-explored. Edges in $\Hh$ may get $p$-explored, and vertices in $\Hh$ may get $s$-explored. No vertex or edge will get explored more than once. In particular, no edge of $\Ghat$ will get $p$- \emph{and} $s$-explored.}

\saut

\newcommand{\condp}{(A)~}
\newcommand{\condpp}{(A)}
\newcommand{\conds}{(D)~}
\newcommand{\condother}{(B)}
\newcommand{\condothers}{(B)~}
\newcommand{\condcon}{(C)}
\newcommand{\condcons}{(C)~}
\newcommand{\cond}{(E)~}

{For every $\ell>0$, during Step $\ell$, we will define inductively a sequence $(C_{\ell,n})_n$ of subsets of $V(\mathcal{H})$ and a sequence $(C'_{\ell,n})_n$ of subsets of $V(\mathcal{G})$.
At the end of each iteration of the process, it will be the case that the following conditions hold:
\begin{itemize}
\item[\condp] If an edge $e$ in $\Hh$ is $p$-explored, then there is a lift $e'$ of $e$ in $\mathcal{G}$ such that the set of the $p$-explored lifts of $e$ is precisely $\{e'\}\times\{1,\dots,M\}$.

\item[\condothers] If an edge $e$ in $E(\Hh)$ is $p$-unexplored, then all of its lifts are unexplored.

\item[\condcons] Every element of $C_{\ell,n}'$ is connected to $o'$ by an $\eta$-open path.

\item[\conds] For every edge $e$ in $\Hh$ and each lift $e'$ of $e$ in $\mathcal{G}$, the number of $s$-explored edges of the form $(e',k)$ is at most the number of $s$-explored vertices $u$ in $\Hh$ at distance at most $r$ from some endpoint of $e$.

\item[\cond] The map $\pi$ induces a well-defined surjection from $C'_{\ell,n}$ to $C_{\ell,n}$.
\end{itemize}
}

\paragraph{Step 0} Set $C_0=\{o\}$ and $C'_0=\{o'\}$. Initially, nothing is considered to be $p$- or $s$-explored.

\paragraph{Step $2K+1$} Set $C_{2K+1,0}:=C_{2K}$ and $C'_{2K+1,0}:=C'_{2K}$.

While there is an unexplored edge that intersects $C_{2K+1,n}$ in $\Hh$, do the following {(otherwise finish this step)}:

\vspace{0.25cm}

\begin{enumerate}
\item take $e$ to be the smallest such edge, 

\item pick $u$ an endpoint of $e$ in $C_{2K+1,n}$ and call $v$ its other endpoint,

\item pick $e'$ some lift of $e$ intersecting $\pi^{-1}(\{u\})\cap C'_{2K+1,n} \not= \varnothing$,

\item declare $e$ and all $(e',k)$'s to be $p$-explored (they were unexplored before because of Conditions \condp and \condother),

\item for every $k\leq M$, define $\eta_{(e',k)}:=\omega_{(e,k)}$,

\item set $(C_{2K+1,n+1},C'_{2K+1,n+1}):=(C_{2K+1,n},C'_{2K+1,n})$ if all the $(e,k)$'s are $\omega$-closed; otherwise, set $(C_{2K+1,n+1},C'_{2K+1,n+1}):=(C_{2K+1,n}\cup\{v\},C'_{2K+1,n}\cup e')$.
\end{enumerate}

\vspace{0.1cm}

{When this step is finished{, which occurs after finitely or countably many iterations},} set $C_{2K+1}:=\bigcup_n C_{2K+1,n}$ and $C'_{2K+1}:=\bigcup_n C_{2K+1,n}'$.

\paragraph{Step $2K+2$} Set $C_{2K+2,0}:=C_{2K+1}$ and $C'_{2K+2,0}:=C'_{2K+1}$.

While there is at least one $s$-unexplored vertex in $C_{2K+1}$ whose $r$-ball is ``fully open''\footnote{In the sense that for each $\mathcal{H}$-edge inside, \emph{at least} one of its copies in $\Hhat$ is open.} in $\omega$, do the following {(otherwise finish this step)}:

\vspace{0.25cm}

\begin{enumerate}
\item take $u$ to be the smallest such vertex,

\item pick some $x \in C'_{2K+1}\cap \pi^{-1}(\{u\})\not= \varnothing$,

\item \emph{This paragraph is not an algorithmic substep, but gathers a few relevant observations. Call an edge in $\mathcal{G}$ $p$-explored if one (hence every by \condpp) of its copies in $\Ghat$ is $p$-explored. Call a $p$-explored edge of $\mathcal{G}$ open if at least one of its copies is $\eta$-open. Notice that by construction and as the $r$-ball of $u$ is ``fully open'' in $\omega$, all the $p$-explored edges of $\mathcal{G}$ that lie inside $\pi^{-1}(B_r(u))$ are open. Also note that for each edge lying in $Z(x,r)$, Condition~\conds and the value of $M$ guarantee that at least one of its copies in $\Ghat$ has not been $s$-explored. As a result, for every edge in $Z(x,r)$, either all its copies have a well-defined $\eta$-status and one of them is open, or at least one of these copies has a still-undefined $\eta$-status. This is what makes Substep~\ref{substepbis} possible.}

\item For each $p$-unexplored edge $e'$ in $Z(x,r)$, take its $s$-unexplored copy $(e',k)$ in $\Ghat$ of smallest label $k$, set $\eta_{(e',k)}:=\omega'_{(e',k)}$, and switch its status to $s$-explored.\label{substepbis}

\item If all these newly $s$-explored edges are open (so that $Z$ is ``fully $\eta$-open''), then perform this substep. By \condp and the definition of $r$, for every $\mathcal{H}$-edge $e\in S_{r+\frac{1}{2}}(u)$, there is at least one lift $e'$ of $e$ that is adjacent to $Z(x,r)$ and $p$-unexplored: pick the smallest one. By \conds and the value of $M$, one of its copies $(e',k)$ is $s$-unexplored: pick that with minimal $k=:k_e$. Declare all these edges to be $s$-explored and set $\eta_{(e',k_e)}:=\omega'_{(e',k_e)}$. If all these $(e',k_e)$'s are $\omega'$-open, then say that this substep is successful.\label{substeptwobis}

\item Notice that conditionally on everything that happened strictly before the current Substep~\ref{substepbis}, the event ``Substep~\ref{substeptwobis} is performed and successful'' has some (random) probability $q\geq \phat^{|E(Z(x,r))|}\geq \phat^{|E(B_{3r}(x))|} \geq s$. If the corresponding event does no occur, set $\alpha_u:= 0$. If this event occurs, then, independently on $(\omega,\omega')$ and everything that happened so far, set $\alpha_u:=1$ with probability $s/q\leq 1$ and $\alpha_u:=0$ otherwise. Declare $u$ to be $s$-explored.

\item If $\alpha_u=1$, then set $C_{2K+2,n+1}:=C_{2K+2,n}\cup S_{r+1}(u)$ and $C'_{2K+2,n+1}$ to be the union of $C_{2K+2,n}$, $Z(x,r)$, and the $e'$'s of Substep~\ref{substeptwobis}. Notice that Condition~\condcon~continues to hold as in this case $Z$ is ``fully $\eta$-open'' and $\eta$-connected to $C_{2K+2,n}$. Otherwise, set $C_{2K+2,n+1}:=C_{2K+2,n}$ and $C'_{2K+2,n+1}:=C'_{2K+2,n}$.
\end{enumerate}
\vspace{0.3cm}
When this step is finished, set $C_{2K+2}:=\bigcup_n C_{2K+2,n}$ and $C'_{2K+2}:=\bigcup_n C_{2K+2,n}'$.

\paragraph{Step $\infty$} Set $C_\infty:=\bigcup_K C_K$ and $C'_\infty:=\bigcup_K C'_K$. Take $\eta'$ independent of everything done so far, with distribution $\mathsf{Ber}(\phat)^{\otimes E(\Ghat)}$. Wherever $\eta$ is undefined, define it to be equal to $\eta'$. In the same way, wherever $\alpha$ is undefined, toss independent Bernoulli random variables of parameter $s$, independent of everything done so far.

\vspace{0.3 cm}

By construction, $C_\infty$ has the distribution of the cluster of the origin for the $(p,s)$-process on $\Hh$: it is the cluster of the origin of $((\vee_k \omega_{e,k})_e,\alpha)$ which has distribution $\mathsf{Ber}(p)^{\otimes E(\Hh)}\otimes \mathsf{Ber}(s)^{\otimes V(\Hh)}$. Recall that $\vee$ stands for the maximum operator.
Besides, $C'_\infty$ is included in the cluster of $o'$ for $(\vee_k \eta_{e,k})_e$, which is a $p$-bond-percolation on $\mathcal{G}$. Finally, the coupling guarantees that $\pi$ surjects $C'_\infty$ onto $C_\infty$. Proposition~\ref{propone} follows.

\begin{rem}
This construction adapts to site percolation. The lift is the same as in \cite{bspercobeyond} while the ``multiple edges'' trick now consists in defining $\Ghat$ as follows: each vertex has $M$ possible states, and it is $p$-open if one of its $\phat$-states says so.
\end{rem}

\section{Proof of Proposition~\ref{proptwo}}
\label{sec:prooftwo}

In this proof, we follow the strategy of Aizenman and Grimmett \cite{aizenmangrimmett, essrev}.

By monotonicity, we can assume without loss of generality that $s<1$. Let $\theta_L(p,s)$ be the $\mathbb{P}_{p,s}$-probability of the event $\mathcal{E}_L := \{\mathcal{C}_o(\omega,\alpha) \cap S_L(o) \neq \varnothing \}$, and $\theta(p,s)=\lim_{L\rightarrow\infty}\theta_L(p,s)$ be the probability that $\mathcal{C}_o(\omega,\alpha)=\mathcal{C}_{\mathcal{H}}^{p,s}(o)$ is infinite.
We claim that in order to prove Proposition~\ref{proptwo}, we only need to show that for any $\varepsilon>0$, there exist $c=c(\varepsilon)>0$ and $L_0(\varepsilon)\geq1$ such that for any $p,s \in [\varepsilon,1-\varepsilon]$ and $L\geq L_0$, we have
\begin{equation}
\label{diffineq}
 \frac{\partial}{\partial s} \theta_L(p,s) \geq c\frac{\partial}{\partial p} \theta_L(p,s).
\end{equation}
Indeed, assume that (\ref{diffineq}) is true. It is easy to see that, since $p_c(\mathcal{H})\in(0,1)$, for any $s\in(0,1)$, there is some $\varepsilon>0$ such that we can find a curve  --- actually a line segment --- $(\mathbf{p}(t),\mathbf{s}(t))_{t\in[0,s]}$ inside $[\varepsilon,1-\varepsilon]^2$ satisfying $\frac{\mathbf{p}'(t)}{\mathbf{s}'(t)}=-c$ for all $t\in[0,s]$ and $p_0:=\mathbf{p}(0)>p_c(\mathcal{H})$, $p_s:=\mathbf{p}(s)<p_c(\mathcal{H})$, $\mathbf{s}(s)=s$. But now note that (\ref{diffineq}) implies that $t\mapsto\theta_L(\mathbf{p}(t),\mathbf{s}(t))$ is a non-decreasing function for all $L\geq L_0$. In particular we have $\theta(p_s,s)=\theta(\mathbf{p}(s),\mathbf{s}(s)) = \lim_L\theta_L(\mathbf{p}(s),\mathbf{s}(s)) \geq \lim_L\theta_L(\mathbf{p}(0),\mathbf{s}(0)) = \theta(\mathbf{p}(0),\mathbf{s}(0))\geq \theta(p_0,0)>0$,
where in the last inequality we use $p_0>p_c(\mathcal{H})$. By monotonicity, we conclude that for every $p\in[p_s,1]$, we have $\theta(p,s)>0$ as desired.

Now note that since the event $\mathcal{E}_L$, which depends only on finitely many coordinates, is increasing in both $\omega$ and $\alpha$, the Margulis-Russo formula gives us
\[ \frac{\partial}{\partial p} \theta_L(p,s) = \sum_{e}\mathbb{P}_{p,s}(e \text{ is $p$-pivotal for } \mathcal{E}_L), \]
\[ \frac{\partial}{\partial s} \theta_L(p,s) = \sum_{x}\mathbb{P}_{p,s}(x \text{ is $s$-pivotal for } \mathcal{E}_L). \]
Recall that an edge $e$ is said to be \defini{$p$-pivotal} for an increasing event $\mathcal{E}$ in a configuration $(\omega,\alpha)$ if $(\omega\cup\{e\},\alpha)\in\mathcal{E}$ but $(\omega\setminus\{e\},\alpha)\notin\mathcal{E}$. 
Similarly, a vertex $x$ is said to be \defini{$s$-pivotal} for an increasing event $\mathcal{E}$ in a configuration $(\omega,\alpha)$ if $(\omega,\alpha\cup\{x\})\in\mathcal{E}$ but $(\omega,\alpha\setminus\{x\})\notin\mathcal{E}$. 

It follows from the above formulas that in order to derive (\ref{diffineq}), it is enough to prove that for some $R, L_0>0$, for every $\varepsilon>0$, there is some $c'>0$ such that for any edge $e$, any $p,s\in[\varepsilon,1-\varepsilon]$, and any $L\geq L_0$, one has
\begin{equation}
\label{goal}
\sum_{x\in B_R(e)} \mathbb{P}_{p,s}(x \text{ is $s$-pivotal for } \mathcal{E}_L) \geq c' \mathbb{P}_{p,s}(e \text{ is $p$-pivotal for } \mathcal{E}_L),
\end{equation}
where for $e=\{x,y\}$, we set $B_R(e):=B_R(x)\cup B_R(y)$.
Indeed, since each vertex can be in $B_R(e)$ for at most $C:=\max_x |E(B_{R+1}(x))|$ different $e$'s, summing (\ref{goal}) over $e$ gives:
\[\sum_{x} C \mathbb{P}_{p,s}(x \text{ is $s$-pivotal for } \mathcal{E}_L) \geq c' \sum_{e}\mathbb{P}_{p,s}(e \text{ is $p$-pivotal for } \mathcal{E}_L)\]
which implies (\ref{diffineq}) for $c:=c'/C$.

The following deterministic lemma directly implies (\ref{goal}).
 
\begin{figure}[h!]
	\centering
	\begin{subfigure}[b]{0.48\linewidth}
		\includegraphics[width=\linewidth]{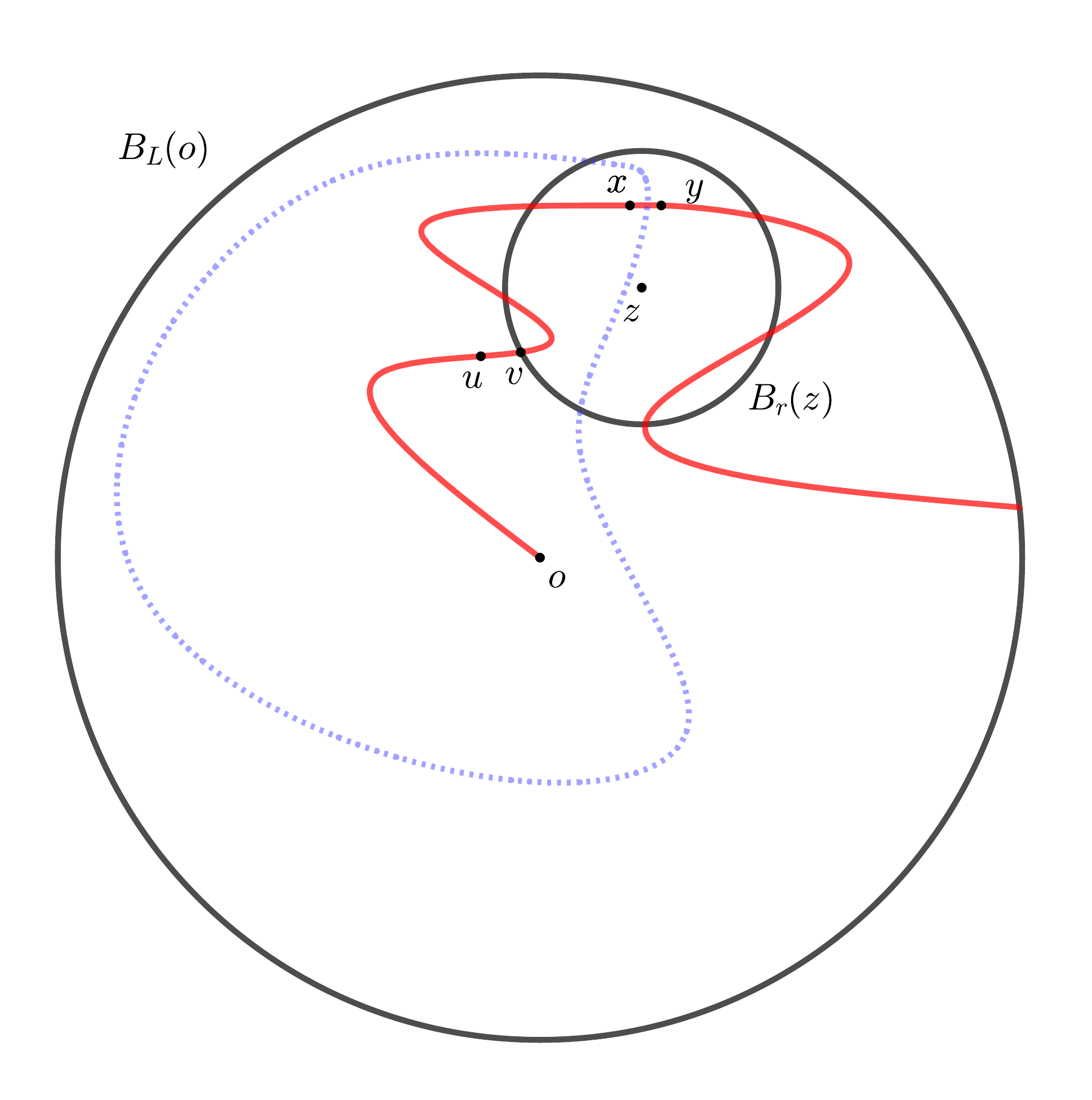}
		\caption*{Configuration $(\omega,\alpha)$.}
	\end{subfigure}
	\begin{subfigure}[b]{0.48\linewidth}
		\includegraphics[width=\linewidth]{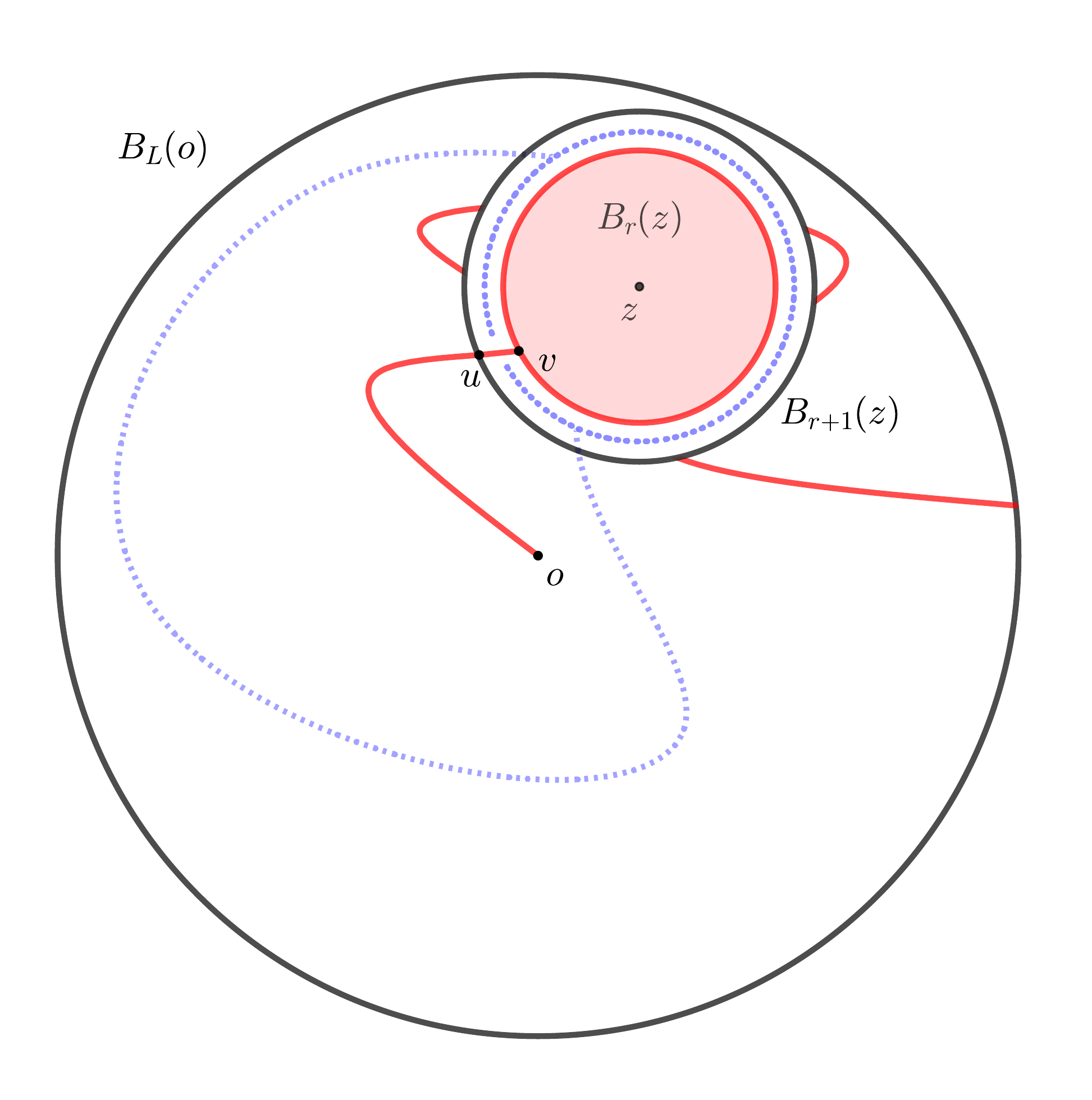}
		\caption*{Configuration $(\omega',\alpha')$.}
	\end{subfigure}
	\caption{A picture of Case a in the proof of Lemma \ref{lem:proptwo}. The colour red represents open edges, either in odd or even steps. The dashed lines in blue represent closed edges preventing certain connections.}
	\label{fig:coffee}
\end{figure}

\begin{lemma}
\label{lem:proptwo}
	There are constants $R$ and $L_0$ such that the following holds. If $L\geq L_0$ and an edge $e$ is $p$-pivotal for $\mathcal{E}_L$ in a configuration $(\omega,\alpha)$, then there exist a configuration $(\omega',\alpha')$ differing from $(\omega,\alpha)$ only inside $B_R(e)$ and a vertex $z$ in $B_R(e)$ such that $z$ is $s$-pivotal for $\mathcal{E}_L$ in $(\omega',\alpha')$.
\end{lemma}

\begin{proof}
 Take $R:=3r+1$ and $L_0:=2r+2$. Let $(\omega,\alpha)$ and $e$ be as in Lemma~\ref{lem:proptwo} and assume without loss of generality that $(\omega,\alpha)\in\mathcal{E}_L$. Now, remove from $\alpha$ all the vertices in $B_R(e)$ one by one. If at some point we get, for the first time, a configuration $(\omega,\alpha')$ that is not in $\mathcal{E}_L$ anymore, then it means that the last vertex $z$ that was removed is $s$-pivotal for that configuration $(\omega,\alpha')$, thus yielding the conclusion of the lemma. Therefore we can assume that $(\omega,\alpha')\in\mathcal{E}_L$ where $\alpha':=\alpha\setminus B_R(e)$. In particular, $e$ is still $p$-pivotal in $(\omega,\alpha')$. We now have two cases.
 
\paragraph{Case a.} \emph{The edge $e=\{x,y\}$ is far from the origin $o$, namely $d(o,e) > r$.}

\vspace{0.08cm}

 Since $e$ is $p$-pivotal for $\mathcal{E}_L$, we have $e\subset B_L(o)$ and $e\not\subset S_L(o)$. So we can assume without loss of generality that $x\in B_{L-1}(o)$. Take $z$ to be a vertex such that $x\in B_r(z)\subset B_{L-1}(o)$ and $o\notin B_r(z)$.\footnote{Just take a suitable vertex in some geodesic from $x$ to $o$. In the case where $d(x,S_{L-1}(o))\geq r$ one can simply take $z=x$. Here we are using that $L\geq L_0=2r+2$.} 
 Now, take some vertex $u\in S_{r+1}(z)$ such that $u\in\mathcal{C}_o(\tilde{\omega},\alpha')$, where $\tilde{\omega}$ is given by closing in $\omega$ all the edges inside $B_{r+1}(z)$, i.e.\ $\tilde{\omega}:=\omega\setminus E(B_{r+1}(z))$. Such a vertex can be obtained as follows. Let $n$ be the first step of the exploration that contains some vertex of $S_{r+1}(z)$, i.e. such that $C_n(\omega,\alpha')\cap S_{r+1}(z)\neq\varnothing$. The previous step $n-1$ does not depend on the state of the edges inside $B_{r+1}(z)$. In particular, one has $C_{n-1}:=C_{n-1}(\omega,\alpha')=C_{n-1}(\tilde{\omega},\alpha')$. Notice that as $\alpha'\cap B_{2r+1}(z)=\varnothing$, the step $n$ is actually an odd one (in which we only explore things in $\omega$). Therefore $C_{n-1}$ is $\omega$-connected to $S_{r+1}(z)$. In particular, there is some $u\in S_{r+1}(z)$ such that $C_{n-1}$ is $\omega$-connected to $u$ outside $ B_{r+1}(z)$, thus also $\tilde{\omega}$-connected. All of this implies that $u\in C_n(\tilde{\omega},\alpha') \subset \mathcal{C}_o(\tilde{\omega},\alpha')$. 
 Let $v$ be any neighbour of $u$ in $B_r(z)$. Finally, define $\omega'$ by opening in $\tilde{\omega}$ the edge $\{u,v\}$ together with all the edges inside $B_r(z)$. Formally, one has $$\omega':=[\omega\setminus E(B_{r+1}(z))] \cup [E(B_{r}(z))\cup \{\{u,v\}\}].$$
 
 \paragraph{Case b.} \emph{The edge $e$ is close to the origin, namely $d(o,e)\leq r$.}

\vspace{0.08cm}

 Without loss of generality, assume  $d(o,x)\leq r$. Then simply take $z=x$ and $\omega'$ given by closing in $\omega$ all the edges inside $B_{r+1}(x)$ and then opening all the edges inside $B_r(x)$, i.e. $\omega':=[\omega\setminus E(B_{r+1}(x))]\cup E(B_r(x))$.
 
\vspace{0.38cm}

 We claim that, in both cases above, $z$ is $s$-pivotal for the event $\mathcal{E}_L$ in the configuration $(\omega',\alpha')$. We are only going to treat Case a. We leave the slightly simpler Case b to the reader.

Remind that by definition of $u$, we have $u\in\mathcal{C}_o(\tilde{\omega},\alpha')$. Since $\alpha'\cap B_{2r+1}(z)=\varnothing$, one can see that after opening at $\tilde{\omega}$ all the edges inside $B_r(z)$ together with $\{u,v\}$ (thus yielding $\omega'$), we do not add any extra vertex in even steps but we add $B_r(z)$ at a certain odd step, so that $\mathcal{C}_o(\omega',\alpha')=\mathcal{C}_o(\tilde{\omega},\alpha')\cup B_r(z)$. In particular, one has $\mathcal{C}_o(\omega',\alpha')\cap S_L(o)=\varnothing$, so that $(\omega',\alpha')\notin \mathcal{E}_L$.

Recall that $z\in \mathcal{C}_o(\omega',\alpha')\subset \mathcal{C}_o(\omega',\alpha'\cup\{z\})$ and that $B_r(z)$ is $p$-open. This implies that $B_{r+1}(z)$ is contained in $\mathcal{C}_o(\omega',\alpha'\cup \{z\})$. Together with $\omega \subset \omega'\cup B_{r+1}(z)$ and $B_{2r+1}(z) \cap \alpha'=\varnothing$, this implies that $\mathcal{C}_o(\omega,\alpha')\subset \mathcal{C}_{B_{r+1}(z)\cup\{o\}}(\omega,\alpha')\subset \mathcal{C}_{B_{r+1}(z)\cup\{o\}}(\omega',\alpha'\cup\{z\})= \mathcal{C}_{o}(\omega',\alpha'\cup\{z\})$. As a result, $\mathcal{C}_o(\omega',\alpha'
 \cup\{z\})\cap S_L(o)\neq\varnothing$, so that $(\omega',\alpha'\cup\{z\})\in\mathcal{E}_L$.
\end{proof}
\vspace{-0.2cm}
\begin{rem}
As in the previous section, the proof above can be adapted to site percolation in a straightforward way.	
\end{rem}

\section{Deriving corollaries from Theorems~\ref{thm:main} and \ref{thm:pu}}
\label{sec:corotothm}

Corollary~\ref{coro:answer} results from and Lemmas~\ref{lem:disjointtrees} and \ref{lem:tamefibres}, and Theorem~\ref{thm:main}, while Corollary~\ref{coro:answerpu} results from and Lemmas~\ref{lem:disjointtrees}, \ref{lem:tamefibres} and \ref{lem:bounded}, and Theorem~\ref{thm:pu}.
Likewise, Corollaries~\ref{coro:strong} and \ref{coro:strongpu} follow by combining Lemma~\ref{lem:strong} with Theorems \ref{thm:main} and \ref{thm:pu}, respectively.

\begin{lemma}
\label{lem:disjointtrees}
Let $\mathcal{G}$ be a graph, and let $G$ be a group acting on $V(\mathcal{G})$ by graph automorphisms. Let $\mathcal{H}$ be the quotient graph $\mathcal{G}/G$ and $\pi:V(\mathcal{G})\to V(\mathcal{H})$ denote the quotient map $x\mapsto Gx$.

If $G\acts V(\mathcal{G})$ is free, then $\pi$ has the disjoint tree-lifting property.
\end{lemma}

\begin{proof}
With the notation of Lemma~\ref{lem:disjointtrees}, let $x$ and $y$ be two distinct vertices of $\mathcal{G}$ such that $\pi(x)=\pi(y)$. Let $\mathcal{T}$ be a subtree of $\mathcal{H}$, and let $\mathcal{T}_x$ be a lift of $\mathcal{T}$ that contains $x$: recall that such a lift exists, as $\pi$ is a weak covering map. As $Gx=Gy$, let us take some $g\in G$ such that $gx=y$. Since $x$ and $y$ are distinct, $g$ is not the identity element. Therefore, by freeness of the action, $g$ has no fixed point.

We claim that $\mathcal{T}_y:=g\mathcal{T}_x$ is a lift of $\mathcal{T}$ that is vertex-disjoint from $\mathcal{T}_x$. It is indeed a lift, as $\forall z\in V(\mathcal{G}),~\pi(z)=\pi(gz)$. To prove vertex-disjunction, let $z\in V(\mathcal{T}_x)\cap gV(\mathcal{T}_x)$. Thus, one can pick $z_\star$ in $V(\mathcal{T}_x)$ such that $z=gz_\star$. As $\pi(z)=\pi(gz_\star)=\pi(z_\star)$, one has $z=z_\star$. Therefore, one has $z=gz$, which contradicts the fact that $g$ has no fixed point.
\end{proof}

\begin{lemma}
\label{lem:tamefibres}
Let $\mathcal{G}$ and $\mathcal{H}$ be quasi-transitive graphs. Let $\pi : V(\mathcal{G})\to V(\mathcal{H})$ be a non-injective weak covering map with the disjoint tree-lifting property.

Then $\pi$ has tame fibres.
\end{lemma}

\begin{proof}
Let $(\mathcal{G},\mathcal{H},\pi)$ satisfy the assumptions of Lemma~\ref{lem:tamefibres}.
First, assume additionally that there is some $r$ such that for every $x\in V(\mathcal{G})$, one has $|B_r(x)|>|B_r(\pi(x))|$. Fix such an $r$. Let $x$ be any vertex of $\mathcal{G}$. As $\pi(B_r(x))= B_r(\pi(x))$, by the pigeonhole principle, one can pick two vertices $y$ and $z$ in $B_r(x)$ such that $\pi(y)=\pi(z)$. Pick a self-avoiding path of length at most $r$ from $\pi(y)$ to $\pi(x)$ in $B_r(\pi(x))$. As $\pi$ has the disjoint tree-lifting property, one can obtain two vertex-disjoint lifts of this path with one starting at $y$ and the other at $z$. Each of these paths ends inside $\pi^{-1}(\{\pi(x)\})\cap B_{2r}(x)$: therefore, this set contains at least one vertex distinct from $x$, thus establishing the tameness of fibres with $R:=2r$.

Let us now prove that the assumptions of the lemma imply the existence of such an $r$. Pick one vertex in each $\mathsf{Aut}(\mathcal{G})$-orbit, thus yielding a finite set $\{x_1,\dots,x_m\}\subset V(\mathcal{G})$. Define $\{u_1,\dots,u_n\}\subset V(\mathcal{H})$ by doing the same in $\mathcal{H}$. Proceeding by contradiction and as $\pi$ is a weak covering map, we may assume that for every $r$, there is some $x\in V(\mathcal{G})$ such that $B_r(x)$ and $B_r(\pi(x))$ are isomorphic as rooted graphs. As a result, for every $r$, there are some $i$ and $j$ such that $B_r(x_i)$ and $B_r(u_j)$ are isomorphic as rooted graphs. As $i$ and $j$ can take only finitely many values, there is some $(i_0,j_0)$ such that for infinitely many values of $r$ --- hence all values of $r$ ---, the rooted graphs $B_r(x_{i_0})$ and $B_r(u_{j_0})$ are isomorphic. It results from local finiteness and diagonal extraction (or equivalently from the fact that the local topology on locally finite connected rooted graphs is Hausdorff) that $\mathcal{G}$ and $\mathcal{H}$ are isomorphic.

This is a contradiction for the following reason. There are two vertices $x$ and $y$ in $\mathcal{G}$ such that $\pi(x)=\pi(y)$: fix such a pair $(x,y)$.  For $r_0$ large enough, for all $i\leq m$, the $r_0$-ball centred at $x_i$ contains $x$ and $y$. Pick such an $r_0$ and pick $i$ such that the cardinality of $B_{r_0}(x_i)$ is minimal: as $\pi(x)=\pi(y)$, the cardinality of $B_{r_0}(\pi(x_i))$ is strictly less than that of $B_{r_0}(x_i)$. Therefore, the minimal cardinality of an $r_0$-ball is not the same for $\mathcal{H}$ and $\mathcal{G}$.
\end{proof}
\vspace{-0.2cm}
\begin{rem}
Notice that in the above proof we only needed to use that we can lift \emph{paths} disjointly.
\end{rem}

\begin{lemma}
\label{lem:strong}
Any strong covering map has the disjoint tree-lifting property.
\end{lemma}

\begin{proof}
Let $\pi:\mathcal{G}\to\mathcal{H}$ denote a strong covering map. Let $x$ and $y$ be two vertices of $\mathcal{G}$ such that $\pi(x)=\pi(y)$. Let $\mathcal{T}$ be a subtree of $\mathcal{H}$, and let $\mathcal{T}_x$ and $\mathcal{T}_y$ be lifts of $\mathcal{T}$ such that $x$ belongs to $V(\mathcal{T}_x)$ and $y$ to $V(\mathcal{T}_y)$. Assume that $V(\mathcal{T}_x)\cap V(\mathcal{T}_y)\not= \varnothing$. Let us prove that $x=y$.

As $\mathcal{T}_x$ is connected, it suffices to prove that if $z_0$ belongs to $V(\mathcal{T}_x)\cap V(\mathcal{T}_y)$, then all its $\mathcal{T}_x$-neighbours belong to $V(\mathcal{T}_x)\cap V(\mathcal{T}_y)$. But this is the case: indeed, any $\mathcal{T}_x$-neighbour $z_1$ of $z_0$ is, by the strong lifting property, the unique neighbour $z_\star$ of $z_0$ such that $\pi(\{z_0,z_\star\})=\pi(\{z_0,z_1\})$, so that $\pi^{-1}(\{\pi(z_1)\})\cap V(\mathcal{T}_y)=\{z_1\}$.
\end{proof}

\saut

In the following lemma, we show that the assumption of bounded fibres in Theorems~\ref{thm:pu} and Corollary~\ref{coro:strongpu} can actually be relaxed to that of \defini{fibres of bounded \emph{cardinality}}, i.e. the condition that $\sup_{u\in V(\mathcal{H})}|\pi^{-1}(\{u\})|<\infty$.

\begin{lemma}
	\label{lem:bounded}
	Let $\mathcal{G}$ and $\mathcal{H}$ be quasi-transitive graphs. Assume that there is a non-injective weak covering map $\pi:V(\mathcal{G})\to V(\mathcal{H})$ with the disjoint tree-lifting property and fibres of bounded cardinality.

	Then there is a map $\pi_\star:V(\mathcal{G})\to V(\mathcal{H})$ satisfying all these conditions and that furthermore has bounded and tame fibres.
\end{lemma}

\begin{rem}
Concerning Corollary~\ref{coro:strongpu}, the boundedness assumption can be relaxed further to the condition that $\pi^{-1}(\{o\})$ is finite. Indeed, for a \emph{strong} covering map, the cardinality of $\pi^{-1}(\{u\})$ does not depend on $u$. 
\end{rem}

\begin{proof}
	First, let us prove that there is a weak covering map $\pi_\star:V(\mathcal{G})\to V(\mathcal{H})$ with the disjoint tree-lifting property and bounded fibres. If $\pi$ has bounded fibres, then we are done. Thus, assume that this is not the case. Let $K$ denote the maximal cardinality of a fibre, i.e.~$K=\max_{u\in V(\mathcal{H})} |\pi^{-1}(\{u\})|$. Since $\pi$ does not have bounded fibres and since $u\mapsto \mathsf{diam}(\pi^{-1}(\{u\})$ is 2-Lipschitz, for every $n$, there is some $x_n\in V(\mathcal{G})$ such that
	$$
	\forall u \in V(\mathcal{H}),~\left|\pi^{-1}(\{u\})\cap B_n(x_n)\right|\leq K-1. 
	$$
As $\mathcal{G}$ is quasi-transitive, one can pick $F$ some finite set of vertices of $\mathcal{G}$ that intersects every $\mathsf{Aut}(\mathcal{G})$-orbit. For every $n$, pick some graph automorphism $\varphi_n$ of $\mathcal{G}$ such that $\varphi^{-1}(x_n)\in F$, and define the equivalence relation $\mathcal{R}_n$ on $V(\mathcal{G})$ by:
$$
x \mathcal{R}_n y \iff \pi\left(\varphi_n(x)\right) = \pi\left(\varphi_n(y)\right).
$$
By taking a pointwise limit of these relations along a converging subsequence, one can endow $V(\mathcal{G})$ with an equivalence relation $\mathcal{R}$ such that:
	\begin{itemize}
		\item $\mathcal{G}/\mathcal{R}$ is isomorphic to $\mathcal{H}$,
		\item the projection $\pi_1 : V(\mathcal{G})\to V(\mathcal{G})/\mathcal{R}$ is a weak covering map with the disjoint tree-lifting property,
		\item every $\mathcal{R}$-class has cardinality at most $K-1$.
	\end{itemize}If $\pi_1$ has bounded fibres, then we are done. Otherwise, iterate the process, applying the same construction to $\pi_1$ instead of $\pi$. Since the maximal cardinality of a fibre cannot decrease forever, this process stops at some suitable $\pi_\star$.
	
	Now, we need to show that $\pi_\star$ has tame fibres. Notice that the weak covering map $\pi_\star$ cannot be injective, as $\mathcal{G}$ and $\mathcal{H}$ are not isomorphic: see the last paragraph of the proof of Lemma~\ref{lem:tamefibres}. As $\pi_\star$ has the disjoint tree-lifting property, every $\pi_\star$-fibre $\pi_\star^{-1}(\{u\})$ has cardinality at least 2. As $\pi_\star$ has bounded fibres, this implies that $\pi_\star$ has tame fibres.
\end{proof}

\section{Proof of Theorems~\ref{thm:nonstrictpu} and \ref{thm:pu}}
\label{sec:pu}

As a warm-up, let us first prove Theorem~\ref{thm:nonstrictpu}.

\subsection{Proof of Theorem~\ref{thm:nonstrictpu}}
\label{subsec:easypu}

In what follows, we will denote by $\Pp_p$ the percolation measure of parameter $p$ on both graphs $\mathcal{G}$ and $\mathcal{H}$, but this will not cause any confusion. For $A$ and $B$ two subsets of the vertices of a graph, we write ``$A \leftrightarrow B$'' for the event that there is an open path intersecting both $A$ and $B$. Similarly, ``$A\leftrightarrow\infty$'' will denote the event that there is an infinite (self-avoiding) open path intersecting $A$.

Let $\mathcal{G}$, $\mathcal{H}$ and $\pi$ be as in Theorem~\ref{thm:nonstrictpu}. The coupling used in \cite{bspercobeyond} to prove the monotonicity of $p_c$ under covering maps yields straightforwardly the following fact: for any two finite subsets $A,B\subset V(\mathcal{H})$ one has
\begin{equation}
\label{similiprop}
\Pp_{p}\big[\pi^{-1}(A)\leftrightarrow \pi^{-1}(B) \big]\geq \Pp_{p}\big[A \leftrightarrow B\big].
\end{equation}

Assume that $p>p_u(\mathcal{H})$. By uniqueness of the infinite cluster at $p$ and the Harris-FKG inequality, one has 
\begin{align*}
\Pp_p[ B_\ell(u) \leftrightarrow B_\ell (v)]&\geq \Pp_p[ B_\ell(u) \leftrightarrow\infty,~ B_\ell (v)\leftrightarrow\infty]\\
&\geq \Pp_p[ B_\ell(u) \leftrightarrow\infty]\Pp_p[ B_\ell(v) \leftrightarrow\infty]
\end{align*} 
for any two vertices $u,v\in V(\mathcal{H})$. This implies, by quasi-transitivity, that
$$\lim_{\ell\rightarrow\infty}~\inf_{u,v \in V(\mathcal{H})} \Pp_p[ B_\ell(u) \leftrightarrow B_\ell (v)]=1.$$
Let $K$ be given by the boundedness of the fibres. As for any vertex $x\in V(\mathcal{G})$, one has $\pi^{-1}(B_\ell(\pi(x)))\subset B_{\ell+K}(x)$, inequality~(\ref{similiprop}) and the previous equation imply that 
$$\lim_{\ell\rightarrow\infty}~\inf_{x,y \in V(\mathcal{G})} \Pp_p[ B_{\ell}(x) \leftrightarrow B_{\ell} (y)]=1.$$ 
Now simply remind that the above equation guarantees that $p\geq p_u(\mathcal{G})$, see \cite{Sch99}.

\subsection{Proof of Theorem~\ref{thm:pu}}

The proof of Theorem~\ref{thm:pu} follows quite closely that of Theorem~\ref{thm:main}.

\saut

Let $\mathcal{G}$, $\mathcal{H}$ and $\pi$ be as in Theorem~\ref{thm:pu}.
Let $r$ be a positive integer.
We use the $(p,s)$-model of Section~\ref{sec:proofmain}, except that we now initialise it at any finite set $A$, instead of just at a single point $o$. When using the $(p,s)$-model initialised at some finite set $A\subset V(\mathcal{H})$, if $B$ is a subset of $V(\mathcal{H})$, we write ``$A \leadsto B$'' for the event ``$\mathcal{C}_{A}\cap B\not= \varnothing$''.

\saut

Here are two propositions, which are reminiscent of Propositions~\ref{propone} and \ref{proptwo}. 

\begin{prop}
\label{proponepu}There is some choice of $r\geq 1$ such that the following holds: for every $\varepsilon >0$, there is some $s\in (0,1)$ such that for every $p\in [\varepsilon, 1]$, for every non-empty finite subset $A'$ of $V(\mathcal{G})$, the random set $\mathcal{C}_{\mathcal{H}}^{p,s}(\pi(A'))$ is stochastically dominated by $\pi\left(\mathcal{C}_{\mathcal{G}}^{p}(A')\right)$. In particular, for any two finite subsets $A,B\subset V(\mathcal{H})$, one has
$$\Pp_{p}\big[\pi^{-1}(A)\leftrightarrow \pi^{-1}(B) \big]\geq \Pp_{p,s}\big[A \leadsto B \big].$$
\end{prop}

Given a positive integer $r$, we say that a subset $B$ of $V(\mathcal{H})$ is \defini{$\mathbf{r}$-nice} if it is finite, non-empty, and if for every $u \in V(\mathcal{H})\backslash B$, there is some vertex $v$ of $\mathcal{H}$ such that $B_r(v)$ contains $u$ and does not intersect $B$.

\begin{prop}
\label{proptwopu}
For every $r\geq 1$ and $s,\epsilon>0$, there exists $\delta>0$ such that the following holds: for every $p\in [\epsilon,1-\epsilon]$ and any two non-empty finite subsets $A, B\subset V(\mathcal{H})$ such that $B$ is $r$-nice and  $d(A,B)>3r$,~\footnote{recall that $d(A,B):=\min \{d(u,v):u\in A, v\in B\}$}  one has
$$
\Pp_{p,s}\big[A \leadsto B \big]\geq \Pp_{p+\delta}\big[A \leftrightarrow B\big].
$$
\end{prop}

\saut

Proposition~\ref{proponepu} is proved exactly as Proposition~\ref{propone}, except that the process is initialised at $(A',\pi(A'))$ instead of $(\{o'\},\{o\})$. Recall that the assumptions of Theorem~\ref{thm:pu} imply that $\pi$ has tame fibres. 

In Section~\ref{subsec:two}, we explain  how to adjust the proof of Proposition~\ref{propone} in order to get Proposition~\ref{proptwopu}.

\vspace{0.44cm}

\begin{proofpu}
If $p_u(\mathcal{H})=1$, then the conclusion holds trivially, so we can assume that $p_u(\mathcal{H})<1$.
Since in addition $p_u(\mathcal{H})\geq p_c(\mathcal{H})>0$, we can find some $\varepsilon>0$ such that $p_u(\mathcal{H})\in (\epsilon,1-\epsilon)$.
By Proposition~\ref{proponepu}, we can pick $r\in \N$ and $s\in (0,1)$ such that for every $p\in [\varepsilon, 1]$, for any two non-empty finite subsets $A,B$ of $V(\mathcal{H})$, one has 
$$
\Pp_{p}\big[\pi^{-1}(A)\leftrightarrow \pi^{-1}(B) \big]\geq \Pp_{p,s}\big[A \leadsto B \big].
$$

By applying Proposition~\ref{proptwopu} to some parameter $p\in (\epsilon,1-\epsilon)$ that satisfies $p<p_u(\mathcal{H})<p+\delta=:q$, we get that for any two non-empty finite subsets $A, B\subset V(\mathcal{H})$ such that $B$ is $r$-nice and $d(A,B)>3r$, one has
$$
\Pp_{p,s}\big[A \leadsto B \big]\geq \Pp_{q}\big[A\leftrightarrow B \big].
$$

Let $K$ be given by the fact that $\pi$ has bounded fibres. 
Notice that for every $x,y\in V(\mathcal{G})$, one has $d(x,y)-K\leq d(\pi(x),\pi(y))\leq d(x,y)$. 
Let $\ell$ be a positive integer and $x, y$ be vertices of $\mathcal{G}$ such that $d(x,y)>L(\ell):=2\ell+4r+K$. Define $u:=\pi(x)$, $v:=\pi(y)$, $A:= B_\ell(u)$ and $B:=V(\mathcal{H})\backslash \bigcup_{w: ~d(w,v)>r+\ell} B_r(w)$. Since $B$ is $r$-nice and $d(A,B)>3r$, we have 
$$
\Pp_{p}\big[\pi^{-1}(A)\leftrightarrow \pi^{-1}(B) \big]\geq \Pp_{p,s}\big[A \leadsto B \big]\geq \Pp_{q}\big[A\leftrightarrow B \big].
$$
Also notice that $B_{\ell}(v)\subset B \subset B_{\ell+r}(v)$, $\pi^{-1}(A)\subset B_{\ell+K}(x)\subset B_L(x)$ and $\pi^{-1}(B)\subset B_{\ell+r+K}(y)\subset B_L(y)$. These inclusions combined with the previous inequality give
$$\Pp_{p}\big[B_{L(\ell)}(x)\leftrightarrow B_{L(\ell)}(y) \big]\geq \Pp_{q}\big[B_{\ell}(\pi(x))\leftrightarrow B_{\ell}(\pi(y)) \big]$$
for any two vertices $x,y\in V(\mathcal{G})$ such that $d(x,y)>L(\ell)$. Notice that this inequality is still true when $d(x,y)\leq L(\ell)$, as the left hand side is then equal to 1.
Taking the infimum over $x,y \in V(\mathcal{G})$ and then sending $\ell$ to infinity gives
$$\lim_{L\rightarrow\infty}\inf_{x,y\in V(\mathcal{G})}\Pp_{p}\big[B_{L}(x)\leftrightarrow B_{L}(y) \big]\geq 
\lim_{\ell\rightarrow\infty}\inf_{u,v\in V(\mathcal{H})}\Pp_{q}\big[B_{\ell}(u)\leftrightarrow B_{\ell}(v) \big]=
1$$
where the last equality follows, as in Section~\ref{subsec:easypu}, from the fact that $q>p_u(\mathcal{H})$. It follows from the above equation (see \cite{Sch99}) that $p_u(\mathcal{G})\leq p< p_u(\mathcal{H})$.
\end{proofpu}

\begin{rem}
A recent paper of Tang \cite{tang} proves that on any quasi-transitive graph, there is a unique infinite cluster at parameter $p$ if and only if $\inf_{u,v\in V} \Pp_p[u \leftrightarrow v]>0$. By using this theorem instead of \cite{Sch99}, one can simplify the above proof: one does not need to connect large balls anymore, but only vertices.
\end{rem}

\subsection{Proof of Proposition~\ref{proptwopu}}
\label{subsec:two}

The proof follows the same lines as that of Proposition \ref{proptwo}, so we will only highlight the necessary adaptations here. 

For any two finite subsets $A,B\subset V(\mathcal{H})$, we consider the following finite dimensional approximation of the event $A\leadsto B$: for each $L$, define $\mathcal{E}^{A,B}_{L}:=\{(\omega,\alpha): \mathcal{C}_A(\omega_{L},\alpha_{L})\cap B\neq \varnothing\}$, where $\omega_{L}$ (resp. $\alpha_{L}$) is the configuration equal to $\omega$ (resp. $\alpha$) in $B_L(o)$ and equal to $0$ elsewhere. 
By the argument presented at the beginning of Section \ref{sec:prooftwo}, one can easily reduce the proof to the following deterministic lemma.

\begin{lemma}
	\label{lem:proptwopu}
	There is a constant $R$ such that the following holds. For any two non-empty finite subsets $A, B\subset V(\mathcal{H})$ such that $B$ is $r$-nice and $d(A,B)>3r$, there is some $L_0=L_0(A,B)$ such that for all $L\geq L_0$, if an edge $e$ is $p$-pivotal for $\mathcal{E}^{A,B}_L$ in a configuration $(\omega,\alpha)$, then there exist a configuration $(\omega',\alpha')$ differing from $(\omega,\alpha)$ only inside $B_R(e)$ and a vertex $z$ in $B_R(e)$ such that $z$ is $s$-pivotal for $\mathcal{E}^{A,B}_L$ in $(\omega',\alpha')$.
\end{lemma}

\begin{proof}
	As in Lemma \ref{lem:proptwo}, it is enough to take $R=3r+1$. Given $A$ and $B$ as above, take $L_0$ such that $A\cup B \subset B_{L}(o)$ and $d(A\cup B,S_{L}(o))>3r$ for all $L\geq L_0$.
	Let $(\omega,\alpha)$ and $e$ be as in Lemma~\ref{lem:proptwopu}. As before, we can assume that $e$ is $p$-pivotal for $\mathcal{E}^{A,B}_L$ in $(\omega,\alpha')$, where $\alpha':=\alpha\setminus B_R(e)$. Again, we have two cases.
	 
	 \paragraph{Case a.} \emph{The edge $e=\{x,y\}$ is far from $A$, namely $d(e,A) > r$.}
	 
	 \vspace{0.08cm}
	 
	 Notice that, since $e$ is $p$-pivotal, we can assume without loss of generality that $x\notin B$. In this case, one can always find a vertex $z$ such that $B_r(z)\subset B_L\setminus(A\cup B)$ and $x\in B_r(z)$. Indeed, if $d(x,B)>r$ and $d(x,S_{L}(o))\geq r$, it suffices to take $z=x$; if $d(x,B)\leq r$, we use the fact that $B$ is $r$-nice to find $z$ such that $B_r(z)\cap B=\varnothing$ and $x\in B_r(z)$, which directly implies $B_r(z)\subset B_L(o)\setminus A$ since $d(B,S_{L}(o))>3r$ and $d(A,B)>3r$; and if $d(x,S_{L}(o))< r$, we can take an appropriate $z$ in the geodesic path from $o$ to $x$ in such a way that $x\in B_r(z)\subset B_L$, which directly implies $B_r(z)\cap (A\cup B)=\varnothing$ since $d(A\cup B,S_{L}(o))> 3r$. As in the proof of Lemma \ref{lem:proptwo}, we can find $u\in S_{r+1}(z)$ such that $u\in\mathcal{C}_o(\tilde{\omega},\alpha')$, where $\tilde{\omega}:=\omega\setminus E(B_{r+1}(z))$. Pick $v\in B_r(z)$ some neighbour of $u$ and define $\omega':=[\omega\setminus E(B_{r+1}(z))] \cup [E(B_{r}(z))\cup \{\{u,v\}\}]$.
	 
	 \paragraph{Case b.} \emph{The edge $e=\{x,y\}$ is close to $A$, namely $d(e,A)\leq r$.}
	 
	 \vspace{0.08cm}
	 
	 Without loss of generality, assume  $d(x,A)\leq r$. Then simply take $z=x$ and $\omega'$ given by closing in $\omega$ all the edges inside $B_{r+1}(x)$ and then opening all the edges inside $B_r(x)$, i.e. $\omega':=[\omega\setminus E(B_{r+1}(x))]\cup E(B_r(x))$.
	 
	 \vspace{0.38cm}
	
	 One can check in the same way as in the proof of Lemma \ref{lem:proptwo} that in both cases above, $z$ is $s$-pivotal for the event $\mathcal{E}^{A,B}_{L}$ in the configuration $(\omega',\alpha')$.
\end{proof}

\paragraph{Acknowledgements}
{We are indebted to Hugo Duminil-Copin, Vincent Tassion, and Augusto Teixeira for many valuable discussions. We also wish to thank Itai Benjamini for suggesting us to investigate the strict monotonicity question for $p_u$. We are grateful to Vincent Beffara, Hugo Duminil-Copin, and Aran Raoufi for comments on an earlier version of this paper. SM is funded by the ERC grant GeoBrown: he is thankful to his postdoctoral hosts Nicolas Curien and Jean-François Le Gall, the Laboratoire de Mathématiques d'Orsay, and GeoBrown for excellent working conditions. FS is thankful to his PhD advisor Hugo Duminil-Copin and IHES for excellent working conditions as well.}

\vspace{0.4cm}

\begin{small}

\end{small}
\end{document}